\documentclass[12pt]{amsart}
\usepackage{amsfonts}
\usepackage{amssymb}
\usepackage{amsthm}
\usepackage{fullpage}

\theoremstyle{definition}
\newtheorem{thm}{Theorem}[section]
\newtheorem{prop}[thm]{Proposition}
\newtheorem{lem}[thm]{Lemma}
\newtheorem{cor}[thm]{Corollary}
\newtheorem{defn}[thm]{Definition}
\newtheorem{exmp}[thm]{Example}

\theoremstyle{remark}
\newtheorem{rem}[thm]{Remark}
\newtheorem{rems}[thm]{Remarks}

\newcommand{\inverse}{^{-1}}
\newcommand{\lieg}{\mathfrak{g}}
\newcommand{\lieh}{\mathfrak{h}}
\newcommand{\DD}{\mathcal{D}}
\newcommand{\NN}{\mathbb{N}}
\newcommand{\ZZ}{\mathbb{Z}}
\newcommand{\RR}{\mathbb{R}}
\newcommand{\Aff}{\mathbb{A}}
\newcommand{\norm}[1]{\|#1\|}
\newcommand{\normdot}{\|\cdot\|}
\newcommand{\limitx}{\lim_{t\to 0} \lambda(t)\cdot x}
\newcommand{\limitg}{\lim_{t\to 0} \lambda(t)g\lambda(t)\inverse}
\newcommand{\Char}{\mathrm{char}}
\newcommand{\Lie}{\mathrm{Lie}}
\newcommand{\Ad}{\mathrm{Ad}}
\newcommand{\im}{\mathrm{Im}}
\newcommand{\GL}{\mathrm{GL}}

\title[Optimal Subgroups]
{Optimal Subgroups and Applications to Nilpotent Elements}

\author{Michael Bate}
\address{Christ Church College \\ Oxford University \\ Oxford, OX1 1DP, UK}
\email{bate@maths.ox.ac.uk}

\begin{document}

\renewcommand\thesubsection{\arabic{subsection}}

\begin{abstract}
Let $G$ be a reductive group acting on an affine variety $X$,
let $x\in X$ be a point whose $G$-orbit is not closed,
and let $S$ be a $G$-stable closed subvariety of $X$ which meets
the closure of the $G$-orbit of $x$ but does not contain $x$.
In this paper, we study G.R. Kempf's optimal class $\Omega_G(x,S)$
of cocharacters of $G$ attached to the point $x$;
in particular, we consider how this optimality transfers
to subgroups of $G$.

Suppose $K$ is a $G$-completely reducible subgroup of $G$ which
fixes $x$, and let $H = C_G(K)^0$.
Our main result says that the $H$-orbit of $x$ is also not closed, and
the optimal class $\Omega_H(x,S)$ for $H$ simply consists of the
cocharacters in $\Omega_G(x,S)$ which evaluate in $H$.
We apply this result in the case that $G$ acts on its Lie algebra
via the adjoint representation to obtain some new information about cocharacters
associated with nilpotent elements in good characteristic.
\end{abstract}

\maketitle

\section{Introduction}\label{sec:intro}
Suppose $G$ is a connected reductive linear algebraic group acting on an affine
variety $X$, and $x \in X$ is a point whose $G$-orbit is not closed in $X$.
Let $G\cdot x$ denote the $G$-orbit of $x$, and $\overline{G\cdot x}$ denote its closure in $X$.
A well known result in Geometric Invariant Theory, the Hilbert--Mumford Theorem
\cite[Thm.\ 1.4]{kempf},
states that there exists a cocharacter (one-parameter subgroup) $\lambda$ of $G$
which takes us from $G\cdot x$ to a point in $\overline{G\cdot x}\setminus G\cdot x$.
In \cite{kempf}, G.R. Kempf strengthened this result to provide a so-called \emph{optimal class}
of such cocharacters which enjoys a number of useful properties;
roughly speaking, these cocharacters are the ones which transport us as quickly as
possible outside the orbit of $x$.
The purpose of this paper is to study how these optimal classes behave under passing to
subgroups of $G$; for a subgroup $H$ of $G$ and $x \in X$,
we give some conditions under which the optimal class of cocharacters for $x$ in $H$ is the
set of optimal cocharacters for $x$ in $G$ which evaluate in $H$
(this happens when at least one of the cocharacters
evaluates in $H$, Proposition \ref{prop:gotosubgroup})
and provide a class of subgroups $H$ for which these conditions
are satisfied (Theorem \ref{thm:Gvscent}).

The initial motivation for studying this problem came from the theory of
\emph{associated cocharacters} for nilpotent elements in the Lie algebra
$\lieg$ of $G$.
This theory has been developed by, among others, Jantzen \cite{jantzen}, Premet \cite{premet},
and McNinch \cite{mcninch}, as
a way of replacing the $\mathfrak{sl}_2$-triples which help classify
nilpotent orbits in characteristic zero.
If we have a subgroup $H$ of $G$ and a nilpotent element $e \in \lieh = \Lie(H)$,
a particular problem in this area is to identify whether or not
the set of cocharacters of $H$ associated with $e$ is just the set of cocharacters of
$G$ associated with $e$ which evaluate in $H$, see \cite[$\S$5.12]{jantzen}, \cite{FR}.
In good characteristic, it has been shown that an associated cocharacter is one
of Kempf's optimal cocharacters, but not vice versa.
However, using our general results on optimality,
we show that understanding how the full class of optimal cocharacters behave
with respect to subgroups is enough to also tackle the problem for associated cocharacters
(Theorem \ref{thm:equality}), and provide a wide class of
subgroups which satisfy the relevant properties.
This gives a uniform approach to many results already in the literature, and
also allows some new constructions of interest.

\medskip

We now indicate the layout of the paper.
In Section \ref{sec:prelims}, we begin with some general notation
and definitions for algebraic groups.
In particular, we recall Serre's notion of $G$-complete reducibility
for subgroups of a reductive group, which plays a part in the ensuing discussion.

In Section \ref{sec:kempf}, we recall the results we need from
Kempf's paper \cite{kempf}.
In doing this, we mostly try to use his notation and terminology,
which hopefully makes cross-referencing easier for the interested reader.
In particular, in introducing his results,
we work in a scheme-theoretic context, as this is the way in which Kempf treats
the topic.

In Section \ref{sec:general}, we indicate how Kempf's results can be used
to transfer optimality to subgroups (Proposition \ref{prop:gotosubgroup}),
and how this ties in with $G$-complete reducibility (Theorem \ref{thm:Gvscent}).
Again, we work in quite a general way in this section, mirroring the set-up that Kempf
has in his paper.

In Section \ref{sec:nilpotent}, we apply our general results to the special case of
the adjoint action of $G$ on its Lie algebra.
Here we show that understanding the transfer of optimality to subgroups is enough to understand
how associated cocharacters behave (Theorem \ref{thm:equality}), and
apply this result to give a wide class of subgroups which behave well in this regard
(Corollary \ref{cor:centcase}).

In the final section, we indicate a few ways in which our results can be extended by
considering non-connected groups \'a la \cite{BMR}, \cite{BMR2};
this allows outer automorphisms of reductive groups to come into play.

\section{Preliminaries}\label{sec:prelims}

\subsection{Basic Notation}
Our basic reference for linear algebraic groups is \cite{borel}.
Except for briefly in Section \ref{sec:nonconnected},
$G$ is a connected reductive linear algebraic group defined over
an algebraically closed field $k$ of characteristic $p \geq 0$.
For a subgroup $H$ of $G$, we let $H^0$ denote the identity component of $H$,
$\DD H$ the derived subgroup
of $H$ and $R_u(H)$ the unipotent radical of $H$.
The centralizer of $H$ in $G$ is denoted $C_G(H)$ and $N_G(H)$ is the normalizer of $H$ in $G$.
The Lie algebra of $G$ (resp. $H$) is denoted $\lieg$ (resp. $\lieh$).

A \emph{$G$-scheme} $X$ is a separated $k$-scheme of finite type on which
$G$ acts morphically;
a key example in this paper is the action of $G$ on $\lieg$ via the
adjoint representation $\Ad: G \to \GL(\lieg)$.
A subscheme $S$ of $X$ is called a \emph{$G$-subscheme} if $S$  is a $G$-scheme
and the immersion $S \subseteq X$ is $G$-equivariant.
In this paper, we are interested only in $k$-points of such $k$-schemes,
and we use the notation $x \in X$ as shorthand for ``$x$ is a $k$-point of $X$''.
For $x \in X$, $C_G(x)$ denotes the stabilizer of $x$ in $G$,
$G\cdot x$ denotes the $G$-orbit of $x$ in $X$, and $\overline{G\cdot x}$ denotes the
closure of this orbit.
For a subgroup $H$ of $G$, we let $X^H$ denote the fixed points of $H$ in $X$.

Let $\Psi = \Psi(G,T)$ denote the set of roots of $G$
with respect to a maximal torus $T$.
Fix a Borel subgroup $B$ of $G$ containing $T$ and let
$\Sigma = \Sigma(G, T)$
be the set of simple roots of $\Psi$ defined by $B$. Then
$\Psi^+ = \Psi(B)$ is the set of positive roots of $G$.
For $\beta \in \Psi^+$ write
$\beta = \sum_{\alpha \in \Sigma} c_{\alpha\beta} \alpha$
with $c_{\alpha\beta} \in \mathbb N_0$.
A prime $p$ is said to be \emph{good} for $G$
if it does not divide $c_{\alpha\beta}$ for any $\alpha$
and $\beta$, and \emph{bad} otherwise.
A prime $p$ is good for $G$ if and only if it is good for
every simple factor of $G$, \cite{SS};
the bad primes for the simple groups are $2$ for all groups except type $A_n$,
$3$ for the exceptional groups and $5$ for type $E_8$.
We say $p = \Char k$ is good for $G$ if $p = 0$ or $p$ is a good prime for $G$.

A linear algebraic group $\Gamma$ is called \emph{linearly reductive} if
all rational representations of $\Gamma$ are semisimple;
a torus is linearly reductive in any characteristic.

\subsection{Levi Decompositions}
A linear algebraic group $H$ has a \emph{Levi decomposition}
if there exists a closed subgroup $L$ of $H$ such that $H = L \ltimes R_u(H)$.
The subgroup $L$ is called a \emph{Levi subgroup} of $H$.
For example, parabolic subgroups of connected reductive groups always
have Levi decompositions.
The following result is an extension of a result of Richardson \cite[Prop.\ 6.1]{rich0},
and can be found in \cite[Prop.\ 2.3]{FR}.

\begin{prop}\label{prop:levistable}
Let $H$ be a linear algebraic group with a Levi decomposition such
that $R_u(H)$ acts simply transitively on the set of Levi subgroups of $H$.
Suppose $\Gamma$ is a linearly reductive algebraic group acting on $H$ by automorphisms.
Then $H$ has a $\Gamma$-stable Levi subgroup.
\end{prop}

\subsection{$G$-Complete Reducibility}
A subgroup $H$ of $G$ is said to be \emph{$G$-completely reducible}
($G$-cr) if whenever $H$ is contained in a parabolic subgroup $P$ of $G$,
there exists a Levi subgroup $L$ of $P$ containing $H$.
This notion was introduced by Serre as a way of generalizing the notion
of complete reducibility from representation theory, \cite{serre1}, \cite{serre2}.
By \cite[Property 4]{serre1}, if $H$ is a $G$-completely reducible subgroup of $G$, then
$R_u(H) = \{1\}$, i.e., $H$ is reductive.
Another basic result that has bearing on this paper is that a linearly reductive subgroup
of $G$ is always $G$-completely reducible.
This follows from Proposition \ref{prop:levistable}, see also \cite[Lem.\ 2.6]{BMR}.
The following result also plays an important part in what follows,
it is \cite[Prop.\ 3.12, Cor.\ 3.17]{BMR}.

\begin{prop}
Let $H$ be a $G$-completely reducible subgroup of $G$.
Then $C_G(H)$ is $G$-completely reducible.
In particular, $C_G(H)$ is reductive, and $C_G(H)^0$ is connected reductive.
\end{prop}

In \cite{BMR}, it is shown that the notion of $G$-complete reducibility
has a geometric interpretation,
obtained by considering the diagonal action of $G$ on $G^n$ for various $n \in \NN$.
Our results in Section \ref{sec:general} below
show that the theory of $G$-complete reducibility has applications
when one considers actions of $G$ on other affine $G$-schemes.

\section{The Theory of Kempf--Rousseau--Hesselink}\label{sec:kempf}

In this section we recall some of the main definitions and results
from the paper of Kempf \cite{kempf} about optimal cocharacters for
actions of algebraic groups on affine varieties,
see also Rousseau \cite{rou} and Hesselink \cite{hess}.
We take our notation and terminology mainly from \cite{kempf}, and
the reader should refer there for more detail; we
recall only the parts of the exposition which are important for this paper.
It is important to note, however, that some of the motivation for the results in this
paper came from reading \cite{hess}, so this is also a key reference.

\subsection{Cocharacters and Length Functions}
Let $\mathbb{G}_m$ denote the multiplicative group, whose $k$-points are isomorphic to the
multiplicative group $k^*$ of $k$.
A \emph{cocharacter} of $G$ is a homomorphism of algebraic groups $\lambda:\mathbb{G}_m \to G$.
We let $Y(G)$ denote the set of cocharacters of $G$.
For every maximal torus $T$ of $G$, the subset $Y(T)$ is a free abelian group
of rank equal to the dimension of $T$;
$Y(G)$ is the union of these groups as $T$ ranges over the maximal tori of $G$.
For any $n \in \ZZ$ and $\lambda \in Y(G)$, we define $n\lambda \in Y(G)$ by
$(n\lambda)(t) = \lambda(t^n) = \lambda(t)^n$.
A cocharacter $\lambda\in Y(G)$ is called \emph{indivisible}
if there is no $\mu\in Y(G)$ with $\lambda = n\mu$ for some $n \in \NN$.
There is a left action of $G$ on $Y(G)$; for $g \in G$, $\lambda \in Y(G)$, let
$(g\cdot\lambda)(t) = g\lambda(t)g\inverse$.

We need the concept of a length function $\normdot$ on $Y(G)$
defined on \cite[p305]{kempf}.

\begin{defn}\label{defn:length}
A \emph{length function} on $Y(G)$ is a
function
\[
\normdot:Y(G) \to \RR
\]
such that:
\begin{itemize}
\item[(a)] $\norm{g\cdot\lambda} = \norm{\lambda}$ for all $g \in G$, $\lambda \in Y(G)$;
\item[(b)] for any maximal torus $T$ of $G$, there is a positive definite integer-valued bilinear
form $(\ ,\ )$ on $Y(T)$ such that $\norm{\lambda}^2 = (\lambda,\lambda)$ for all $\lambda \in Y(T)$.
\end{itemize}
\end{defn}

To show length functions exist, one can start with any maximal torus of $G$ and any
positive definite integer-valued bilinear form on $Y(T)$.
Averaging this form over the Weyl group $W \simeq N_G(T)/T$, we obtain a $W$-invariant bilinear
form $(\ ,\ )$, say.
Now for each $\lambda \in Y(G)$, there exists $g \in G$ such that $g\cdot\lambda \in Y(T)$;
we define $\normdot$ by
$\norm{\lambda}^2 = (g\cdot\lambda,g\cdot\lambda)$.
This is well-defined because $(\ ,\ )$ is $W$-invariant.

\subsection{Cocharacters and $G$-Actions}
Let $X$ be an affine $G$-scheme.
For each $\lambda \in Y(G)$ and each $x \in X$, there is a morphism
$\phi_{\lambda,x}:\Aff^1\setminus\{0\} \to X$ given by $\phi_{\lambda,x}(t) = \lambda(t)\cdot x$.
If $\phi_{\lambda,x}$ extends to a morphism from all of $\Aff^1$ to $X$, then we say that
$\limitx$ \emph{exists}.

Now fix $x \in X$ and let $S$ be a closed
$G$-subscheme of $X$ with $x \notin S$.
Let $|X,x|$ denote the set of $\lambda \in Y(G)$ for which $\limitx$ exists.
Then for each $\lambda \in |X,x|$, we have a morphism $M(\lambda):\Aff^1 \to X$
given by the unique extension of $\phi_{\lambda,x}$;
i.e.  $M(\lambda)(t) = \phi_{\lambda,x}(t)$ for all $t \in \Aff^1\setminus\{0\}$,
and $M(\lambda)(0) = \limitx$.
Following \cite[p308]{kempf}, we define $\alpha_{S,x}(\lambda)$ to be the degree of the divisor
$M(\lambda)\inverse(S)$ on $\Aff^1$; this is a non-negative integer,
which is positive if and only if $M(\lambda)(0) \in S$ (\cite[Lem.\ 3.1]{kempf}, see also \cite[Sec.\ 2]{hess}).

\subsection{Cocharacters and Parabolic Subgroups}
To each $\lambda \in Y(G)$, we can associate a parabolic subgroup $P_\lambda$ of $G$,
which consists of the points of $G$ for which $\limitg$ exists.
The unipotent radical of $P_\lambda$ consists of the points for which
$\limitg = 1$, and the centralizer of the image of $\lambda$ in $G$ is a
Levi subgroup of $P_\lambda$, denoted $L_\lambda$.

If $H$ is a reductive subgroup of $G$, and $\lambda \in Y(H)$, we can
associate a parabolic subgroup of $G$ and of $H$ to $\lambda$.
We reserve the notation $P_\lambda$ for parabolic subgroups of $G$,
and let $P_\lambda(H)$ denote the corresponding subgroup of $H$.
Note that $P_\lambda(H) = P_\lambda \cap H$
and $R_u(P_\lambda(H)) = R_u(P_\lambda) \cap H$.

\subsection{Optimality}
The following is \cite[Thm\ 3.4]{kempf}.

\begin{thm}\label{thm:kempf}
Let $X$ be an affine $G$-scheme, and let $x \in X$ be a point
such that $G\cdot x$ is not closed.
Let $S$ be a closed $G$-subscheme of $X$ such that $x \notin S$ and $S$ meets the closure of
$G\cdot x$.
Then, for a fixed choice of length function $\normdot$ on $Y(G)$:
\begin{itemize}
\item[(a)] $\alpha_{S,x}(\lambda)/\norm{\lambda}$ has a maximum value on the
set of non-trivial cocharacters in $|X,x|$;
\item[(b)] let $\Omega_G(x,S)$ denote the set of \emph{indivisible} cocharacters which achieve the
maximum value from (a). Then:
\begin{itemize}
\item[(i)] $\Omega_G(x,S)$ is non-empty;
\item[(ii)] there is a parabolic subgroup $P(x,S)$ of $G$ such that $P_\lambda = P(x,S)$ for
all $\lambda \in \Omega_G(x,S)$;
\item[(iii)] $R_u(P(x,S))$ acts simply transitively on $\Omega_G(x,S)$;
\item[(iv)] any maximal torus of $P(x,S)$ contains a unique member of $\Omega_G(x,S)$.
\end{itemize}
\end{itemize}
\end{thm}
We call $\Omega_G(x,S)$ the \emph{optimal class} of cocharacters for $x \in X$
(with respect to $S$);
in what follows, we only consider optimal classes which are non-empty,
as in Theorem \ref{thm:kempf}.
We refer to the parabolic $P(x,S)$ as the
\emph{optimal parabolic subgroup}.
The optimal nature of $\Omega_G(x,S)$ allows the following corollary, \cite[Cor.\ 3.5]{kempf}.

\begin{cor}\label{cor:kempf}
Let the notation be as in Theorem \ref{thm:kempf}.
Then:
\begin{itemize}
\item[(a)] for any $k$-point $g \in G$, $gP(x,S)g\inverse = P(g\cdot x,S)$;
\item[(b)] $C_G(x) \subseteq P(x,S)$.
\end{itemize}
\end{cor}

\begin{rem}\label{rem:independent}
It is clear that the optimal class depends in general on the choice of
length function $\normdot$ on $Y(G)$, as well as the particular subscheme $S$.
However, Hesselink has shown that for certain $X$, $S$ and $x$, the optimal class is
at least independent of the length function \cite[Thm.\ 7.2]{hess}.
In particular, this happens for the special case we consider in Section \ref{sec:nilpotent},
where $X = \lieg$, $S = \{0\}$ and $x$ is a nilpotent element of $\lieg$, cf. \cite[Ex.\ 7.1(b)]{hess}.

We note here that the set-up in Hesselink's paper is slightly different to Kempf's.
Firstly, he discusses a notion of \emph{uniform instability}, where
instead of a single point $x \in X$, one considers a whole family of points
which are all moved together to the same distinguished point of $X$ (so $S$
is a single point, but $x$ is replaced with a subset of $X$).
Secondly, he uses a norm on $Y(G)$ which is slightly different
from Kempf's length function, and derives an optimal class of so-called
\emph{virtual cocharacters}.
However, his results are easily translated into our setting.
\end{rem}

\section{Transferring Optimality to Subgroups}\label{sec:general}

In this section, we want to consider the following general idea:
For a $G$-scheme $X$ and a subgroup $H$ of $G$,
we get an obvious action of $H$ on $X$ inherited from $G$.
Motivated by \cite[(4.4)]{hess}, we would like to ask whether or not
\begin{equation}\label{eqn:intersection}
\Omega_H(x,S) = \Omega_G(x,S) \cap Y(H),
\end{equation}
for various $x \in X$, $S \subseteq X$.
In order to make sense of this, we have to be slightly careful about the set-up;
in particular, we need to know that $\Omega_H(x,S)$ is defined.

Let $X$ be an affine $G$-scheme, and $H$ a reductive subgroup of $G$.
Suppose $x\in X$ is a point whose $G$- and $H$-orbits are not closed,
and let $S$ be a closed $G$-subscheme of $X$ which
does not contain $x$, but meets the closure of the $H$-orbit of $x$.
These conditions ensure that $S$ is an $H$-subscheme, and $S$ also meets the closure of the
$G$-orbit of $x$.
Since $Y(H) \subseteq Y(G)$, it is clear that $|X,x|_H:= |X,x| \cap Y(H) \subseteq |X,x|$.
It is also clear that for $\lambda \in Y(H)$, the value of $\alpha(S,\lambda)$ does not
depend on whether we consider $\lambda$ as a cocharacter of $H$ acting on the $H$-scheme $X$,
or as a cocharacter of $G$ acting on the $G$-scheme $X$.
The final easy observation that allows us to proceed is the following:

\begin{lem}\label{lem:length}
Let $\normdot$ be a length function on $Y(G)$.
Then the restriction of $\normdot$ to $Y(H)$ is a length function on $Y(H)$.
\end{lem}

\begin{proof}
We need to check conditions (a) and (b) of Definition \ref{defn:length}.
Condition (a) is obvious; if $\normdot$ is $G$-invariant, then it is $H$-invariant.
For (b), let $S$ be a maximal torus of $H$, and let $T$ be a maximal torus of $G$ containing $S$.
Then $Y(S)$ is a subgroup of $Y(T)$, and it is clear that a positive definite
integer-valued bilinear form
on $Y(T)$ restricts to a positive definite integer valued bilinear form on $Y(S)$.
\qed
\end{proof}

The preceding discussion shows that  the optimal class
$\Omega_H(x,S)$ exists and can be defined with respect to the same length function
as $\Omega_G(x,S)$, so our question makes sense
independently of the length function chosen on $Y(G)$.
For the rest of this section, we fix a length function $\normdot$ on $Y(G)$, and
use the restriction of this length function to $Y(H)$ for subgroups $H$ of $G$.
Our next result shows that the desired equality (\ref{eqn:intersection})
holds if and only if $\Omega_G(x,S) \cap Y(H)$ is non-empty.

\begin{prop}\label{prop:gotosubgroup}
Let $X$ be an affine $G$-scheme, let $x \in X$ be such that $G\cdot x$ is
not closed, and let $S$ be a closed $G$-subscheme of $X$ which
does not contain $x$, but meets $\overline{G\cdot x}$.
If $H$ is a reductive subgroup of $G$ such that
$\Omega_G(x,S)\cap Y(H)$ is non-empty,
then $H\cdot x$ is not closed, $S$ meets $\overline{H\cdot x}$ and
$\Omega_H(x,S) = \Omega_G(x,S) \cap Y(H)$.
\end{prop}

\begin{proof}
Let $\lambda \in \Omega_G(x,S)\cap Y(H)$.
Then since $\limitx \in S$ and $\lambda \in Y(H)$,
we have $\limitx \in S\cap \overline{H\cdot x}$.
Since $\limitx$ is not in $G\cdot x$, it is not in $H\cdot x$.
Thus $H\cdot x$ is not closed, and $S$ meets the closure of $H\cdot x$.
In particular, it makes sense to define $\Omega_H(x,S)$ for $x$, $S$ and our fixed length.

Now $\lambda$ is an indivisible
cocharacter of $G$, so $\lambda$ is an indivisible cocharacter of $H$.
Since we use the same length function on $Y(G)$ and $Y(H)$,
and $\alpha(S,\lambda)$ is independent of $G$ and $H$, it is clear
that the maximal value of $\alpha(S,\lambda)/\norm{\lambda}$ as
$\lambda$ runs over $|X,x|_H$ is less than or equal to the maximum value as
$\lambda$ runs over $|X,x|$.
But since our choice of $\lambda$ attains this maximal value, there is equality here,
and since $\lambda$ is an indivisible cocharacter of $H$, this
means $\lambda \in \Omega_H(x,S)$.
Thus $\Omega_G(x,S) \cap Y(H) \subseteq \Omega_H(x,S)$.

Finally, consider $P_\lambda(H) = P(x,S) \cap H$;
this is the optimal parabolic subgroup of $H$ given by $\Omega_H(x,S)$.
Moreover, $R_u(P_\lambda(H)) = R_u(P(x,S)) \cap H$ acts transitively on $\Omega_H(x,S)$.
Since $R_u(P(x,S))$ acts on $\Omega_G(x,S)$,
we see that $\Omega_H(x,S) \subseteq \Omega_G(x,S) \cap Y(H)$, and we are done.
\qed
\end{proof}

Motivated by the previous result, we can now define our notion of optimality for subgroups.
\begin{defn}\label{defn:optimalsubgroup}
Let $H$ be a reductive subgroup of $G$, and let $X$ be an affine $G$-scheme.
Suppose $x\in X$ is a point whose $G$- and $H$-orbits are not closed,
and let $S$ be a closed $G$-subscheme of $X$ which
does not contain $x$, but meets the closure of the $H$-orbit of $x$.
Then, following \cite[(4.4)]{hess}, we call $H$ \emph{optimal for $X$, $x$ and $S$}
if $\Omega_H(x,S) = \Omega_G(x,S) \cap Y(H)$.
\end{defn}

Our next result gives a wide class of optimal subgroups for a general action
of $G$ on an affine variety $X$.

\begin{thm}\label{thm:Gvscent}
Suppose $X$ is an affine $G$-scheme and that $K$ is a $G$-completely reducible subgroup of $G$.
Let $x \in X^K$ be such that $G\cdot x$ is not closed, and let
$S$ be any closed $G$-subscheme of $X$ which meets $\overline{G\cdot x}$ and does not contain $x$.
Set $H = C_G(K)^0$.
Then:
\begin{itemize}
\item[(i)] $H\cdot x$ is not closed;
\item[(ii)] $S$ meets $\overline{H\cdot x}$, and hence $\Omega_H(x,S)$ is defined;
\item[(iii)] $H$ is optimal for $X$, $x$ and $S$.
\end{itemize}
\end{thm}

\begin{proof}
By Proposition \ref{prop:gotosubgroup},
it is enough to show that $\Omega_G(x,S) \cap Y(H)$ is non-empty.
Let $P(x,S)$ be the optimal parabolic subgroup of $G$ from Theorem \ref{thm:kempf}.
Then, by Corollary \ref{cor:kempf}, $C_G(x) \subseteq P(x,S)$.
Since $K \subseteq C_G(x)$, we also have $K \subseteq P(x,S)$.
Now $K$ is $G$-cr, so there exists a Levi subgroup $L$ of $P(x,S)$ containing $K$.
Let $\lambda \in \Omega_G(x,S)$; so $P(x,S) = P_\lambda$,
and $L_\lambda$ is also Levi subgroup of $P(x,S)$.
Then $L$ is $R_u(P(x,S))$-conjugate to $L_\lambda$,
hence $L = L_\mu$ for some $R_u(P(x,S))$-conjugate $\mu$ of $\lambda$.
But $R_u(P(x,S))$ acts on $\Omega_G(x,S)$,
hence $\mu \in \Omega_G(x,S)$.
Finally, we have $K \subseteq L_\mu$, so $K$ centralizes the image of $\mu$, and
$\mu \in Y(C_G(K)^0) = Y(H)$, as required.
\qed
\end{proof}

\begin{rems}\label{rems:endofgeneral}
(i). Note that in Theorem \ref{thm:Gvscent} it is necessary to know that (i) and (ii) hold
to even ask whether $H$ is optimal, as the definition of optimality for subgroups only
applies to those $H$ for which (i) and (ii) hold.

(ii). In characteristic $0$,  a subgroup of $G$ is
$G$-completely reducible if and only if it is reductive \cite[Sec.~2.2]{BMR}.
In this case, part (i) of Theorem \ref{thm:Gvscent} can be deduced from
\cite[$\S$3~Cor.~3.1]{luna}, due to Luna,
and parts (ii) and (iii) are similar to Kempf's result \cite[Cor.~4.5]{kempf}.
In fact, for (i) in characteristic $0$,
the stronger statement that the $G$-orbit of $x \in X^K$ is
closed if and only if the $H$-orbit is closed follows from Luna's result.

(iii). For any $x \in X$, there is a unique closed orbit $C$ in the
closure of $G\cdot x$.
Clearly, $C \subseteq S$ for any closed $G$-subscheme $S$
which meets $\overline{G\cdot x}$.
In applications, it usually suffices to just consider the case $S=C$
when applying Kempf's results.

(iv). We can also interpret our results within the framework of
the paper \cite{hess} of Hesselink (cf. Remark \ref{rem:independent}).
In particular, Theorem \ref{thm:Gvscent} is a generalization of \cite[Prop.\ 9.4]{hess},
which gives the result when $K$ is a torus, so that $H$ is a Levi subgroup
of some parabolic subgroup of $G$ (a \emph{critical subgroup} in \cite{hess}).
\end{rems}

We finish this section with an example which shows that the reverse direction of Theorem
\ref{thm:Gvscent} (and hence also Proposition \ref{prop:gotosubgroup}) does not
work in general.

\begin{exmp}\label{exmp:gcrnotcentcr}
In \cite[Ex.\ 5.1, 5.4]{BMR2}, there are examples of commuting
subgroups $A$ and $B$ of a reductive group $G$ such that
$A$ and $B$ are $G$-cr, but $B$ is not $C_G(A)^0$-cr.
We can find an $n$-tuple $\mathbf{x} = (x_1,\ldots,x_n) \in G^n$
for some $n$ such that $B$ is the closure
of the subgroup generated by $x_1, \ldots, x_n$.

In this situation, if we set $H = C_G(A)^0$,
we know that
$G\cdot \mathbf{x}$ is closed in $G^n$, whereas
$H \cdot \mathbf{x}$ is not, by \cite[Cor.\ 3.7]{BMR}.
Thus any closed $G$-subscheme $S$ which meets $G\cdot\mathbf{x}$ actually contains all
of $G\cdot\mathbf{x}$, and hence all of $\overline{H\cdot\mathbf{x}}$,
so we cannot define $\Omega_G(\mathbf{x},S)$ and $\Omega_H(\mathbf{x},S)$
for such subschemes.
However, the unique closed $H$-orbit $C$ in $\overline{H\cdot\mathbf{x}}$ is
an $H$-subscheme which meets $\overline{H\cdot\mathbf{x}}$ and does not contain $\mathbf{x}$,
so we \emph{can} talk about $\Omega_H(\mathbf{x},C)$.
This example also shows why we have to be very careful
in our set-up when trying to compare $G$- and $H$-actions,
just to make sure that what we are saying even makes sense.
\end{exmp}

\section{Cocharacters Associated with Nilpotent Elements}\label{sec:nilpotent}

In this section, we consider the adjoint action of $G$ on its Lie algebra $\lieg$,
and apply our results to the notion of a cocharacter associated with a nilpotent element of $\lieg$.
Throughout this section, $e$ denotes a nilpotent element in $\lieg$, and
our distinguished closed $G$-subscheme $S$ of $\lieg$ is simply the set $S = \{0\}$,
which is the unique closed $G$-orbit in $\overline{G\cdot e}$.
We keep our fixed length function, and drop $S$ from our notation;
thus $\Omega_G(e)$ is the set of optimal cocharacters for $e$, and $P(e)$
is the corresponding parabolic subgroup of $G$.
If $H$ is a reductive subgroup of $G$, and $e \in \lieh$, then
$\{0\}$ is also the unique closed $H$-orbit in $\overline{H\cdot e}$;
thus $\Omega_H(e)$ is always defined in this situation.

Let $\lambda \in Y(G)$; then $\lambda$ gives a decomposition of $\lieg$
\[
\lieg = \oplus_{i \in \ZZ} \lieg(i,\lambda),
\]
where $\lieg(i,\lambda) := \{x \in \lieg \mid \Ad(\lambda(t)) x = t^i x \textrm{ for all } t \in k^*\}$.

\begin{defn}\label{defn:distinguished}
Let $H$ be a closed subgroup of $G$, with Lie algebra $\lieh \subseteq \lieg$ and
suppose $e \in \lieh \subseteq \lieg$ is a nilpotent element.
Then $e$ is said to be \emph{distinguished} in $\lieh$ if each torus in $C_H(e)$ is contained
in the centre of $H$.
\end{defn}

\begin{defn}\label{defn:associated}
Let $G$ be a connected reductive linear algebraic group, and let $e \in \lieg$ be nilpotent.
A cocharacter $\lambda \in Y(G)$ is called \emph{associated} with $e$ if
\begin{itemize}
\item[(i)] $e \in \lieg(2,\lambda)$;
\item[(ii)] there exists a Levi subgroup $L$ of some parabolic subgroup of $G$ such that
\begin{itemize}
\item[(a)] $e$ is distinguished in $\Lie(L)$;
\item[(b)] $\im(\lambda) \subseteq \DD L$.
\end{itemize}
\end{itemize}
We let $\Omega_G^a(e)$ denote the set of cocharacters of $G$ associated with $e$.
\end{defn}

The following lemma gives a necessary condition for a Levi subgroup $L$
to arise in part (ii) above, see \cite[Rem.\ 4.7]{jantzen}.

\begin{lem}\label{lem:distinL}
Let $e$ be a nilpotent element of $\lieg$, and suppose $L$ is a Levi subgroup of some parabolic subgroup of $G$.
Then $e$ is distinguished in $\Lie(L)$ if and only if $L = C_G(S)$,
where $S$ is a maximal torus of $C_G(e)$.
\end{lem}

Premet \cite{premet} and McNinch \cite{mcninch} have shown that if the
characteristic of $k$ is good for $G$, then there exist cocharacters associated with
any nilpotent $e \in \lieg$.
In fact, in good characteristic, we can say a lot about the set $\Omega_G^a(e)$.
The following is an amalgamation of several results, see \cite[Lem.\ 5.3]{jantzen},
\cite[Thm.\ 2.3, Prop.\ 2.5]{premet},  \cite[Prop.\ 2.11, Prop.\ 2.15, Cor.\ 2.16, Rem.\ 2.17]{FR}.

\begin{prop}\label{prop:keyprop}
Suppose $p$ is good for $G$. Let $e \in \lieg$ be nilpotent.
\begin{itemize}
\item[(i)] $\Omega_G^a(e)$ is a non-empty subset of $\Omega_G(e)$.
\item[(ii)] Let $\lambda \in \Omega_G^a(e)$.
Set $R_e = R_u(P_\lambda) \cap C_G(e)$ and $C_G(e,\lambda) = L_\lambda \cap C_G(e)$.
Then $R_e = R_u(C_G(e))$ and $C_G(e) = R_e \rtimes C_G(e,\lambda)$ is a Levi decomposition
of $C_G(e)$.
\item[(iii)] Any two cocharacters of $G$ associated with $e$ are conjugate by an element of $C_G(e)^0$.
Conversely, $C_G(e)$ acts on $\Omega_G^a(e)$, and $R_e$ acts simply transitively on $\Omega_G^a(e)$.
\item[(iv)] The map $\lambda \to C_G(e,\lambda)$ is a bijection between
$\Omega_G^a(e)$ and the set of Levi subgroups of $C_G(e)$, and this map
is compatible with the action of $R_e$ on both these sets.
\end{itemize}
\end{prop}

Note that since $\Omega_G^a(e) \subseteq \Omega_G(e)$,
the parabolic subgroups $P_\lambda$ in part (ii)
are all equal to the optimal parabolic $P(e)$ given by Kempf's Theorem \ref{thm:kempf}.
Building on the previous section,
if $H$ is a reductive subgroup of $G$ and $e \in \lieh$ is nilpotent,
it is interesting to know when
\begin{equation}\label{eqn:nilpintersection}
\Omega^a_H(e) = \Omega^a_G(e) \cap Y(H),
\end{equation}
see \cite[$\S$5.12]{jantzen}, \cite{FR}.
Again, we have to be slightly careful to ensure that our question makes sense,
which results in some characteristic restrictions in our results.
The first result shows that Eqn.~(\ref{eqn:nilpintersection})
holds if and only if Eqn.~(\ref{eqn:intersection}) holds.

\begin{thm}\label{thm:equality}
Let $H$ be a reductive subgroup of $G$ and $e \in \lieh$ a nilpotent element.
If $p$ is good for $H$ and $G$,
then $\Omega_H(e) = \Omega_G(e) \cap Y(H)$
if and only if $\Omega_H^a(e) = \Omega_G^a(e) \cap Y(H)$.
\end{thm}

\begin{proof}
First suppose that $\Omega_H(e) = \Omega_G(e) \cap Y(H)$.
Since $\Char k$ is good for $H$, we have $\Omega_H^a(e)$ is non-empty and is a subset
of $\Omega_H(e)$.
Let $\lambda \in \Omega_H^a(e)$.
Then $\lambda \in \Omega_G(e)$.
By \cite[Thm.\ 3.17]{FR}, if we can show that $\lambda \in \Omega_G^a(e)$,
then we can conclude that $\Omega_H^a(e) = \Omega_G^a(e) \cap Y(H)$,
as required.

Since $e \in \lieh(2,\lambda) \subseteq \lieg(2,\lambda)$, it is clear that
$\im(\lambda)$ normalizes $C_G(e)$.
Since, by our standing assumption, $p$ is good for $G$,
$C_G(e)$ has a Levi decomposition,
and $R_e$ acts simply transitively on the set of Levi subgroups of $C_G(e)$,
by Proposition \ref{prop:keyprop}.
Moreover, the Levi subgroups of $C_G(e)$ are in bijection with the elements of $\Omega_G^a(e)$.
By Proposition \ref{prop:levistable}, there is a Levi subgroup $C_G(e,\mu)$ of $C_G(e)$ which is stable under
$\im(\lambda)$, where $\mu \in \Omega_G^a(e)$.

Since $\mu$ is associated with $e$, $e \in \lieg(2,\mu)$.
Also, by Lemma \ref{lem:distinL},
there exists a maximal torus $S$ of $C_G(e)$ such that $e$ is
distinguished nilpotent in $\Lie(L)$, where $L = C_G(S)$, and $\im(\mu) \subseteq \DD L$.
Now for any $x = \lambda(t) \in \im(\lambda)$,
since $\im(\lambda)$ normalizes $C_G(e)$, $xSx\inverse$ is a maximal torus of $C_G(e)$;
thus $e$ is distinguished nilpotent in $\Lie(xLx\inverse)$.
Moreover, since $e \in \lieg(2,\lambda)$ also, we have $e \in \lieg(2,x\cdot\mu)$.
Finally, note that $\mu \in Y(\DD L)$ implies $x\cdot\mu \in Y(\DD(xLx\inverse))$.
These arguments suffice to show that for all $x \in \im(\lambda)$,
$x\cdot\mu \in \Omega_G^a(e)$.

Now we have $xC_G(e,\mu)x\inverse = C_G(e, x\cdot\mu) = C_G(e,\mu)$ for all $x \in \im(\lambda)$,
so $\im(\lambda)$ must fix $\mu$.
This means that $\im(\lambda) \subseteq C_G(\im(\mu)) = L_\mu$.
Thus there is a maximal torus $T$ of $L_\mu$ containing $\im(\lambda)$
and $\im(\mu)$; but $\lambda,\mu \in \Omega_G(e)$, so $\lambda = \mu$
by Theorem \ref{thm:kempf}(b)(iv), and we are done.

For the reverse implication, suppose $\Omega_H^a(e) = \Omega_G^a(e) \cap Y(H)$.
Then, since the characteristic is good for $G$ and $H$, this set is non-empty.
Let $\lambda \in \Omega_H^a(e)$.
Then $\lambda \in \Omega_G^a(e) \subseteq \Omega_G(e)$, so
$\lambda \in \Omega_G(e) \cap Y(H)$ and Proposition \ref{prop:gotosubgroup} applies.
\qed
\end{proof}

\begin{rem}
In characteristic $0$, for any reductive subgroup $H$ of $G$
and any nilpotent element $e \in \lieh$, we always have $\Omega_H^a(e) = \Omega_G^a(e) \cap Y(H)$,
see \cite[5.12]{jantzen}.
Thus, by Theorem \ref{thm:equality}, we also always have $\Omega_H(e) = \Omega_G(e) \cap Y(H)$ here.
We are grateful to G. R\"ohrle for pointing out this application.
\end{rem}

The following corollary is immediate from Theorem \ref{thm:Gvscent}
and Theorem \ref{thm:equality}.

\begin{cor}\label{cor:centcase}
Suppose $K$ is a $G$-completely reducible subgroup of $G$, set $H=C_G(K)^0$,
and suppose that $p$ is good for $H$ and $G$.
Then $\Omega_H^a(e) = \Omega_G^a(e) \cap Y(H)$ for all nilpotent elements $e \in \lieh$.
\end{cor}

\begin{rem}
Corollary \ref{cor:centcase} covers in a uniform way many results
already in the literature.
For example, if $K$ is a torus in $G$, then $C_G(K)^0 = C_G(K)$ is
a Levi subgroup of some parabolic of $G$. This gives \cite[Cor.\ 3.22]{FR}.
More generally, if $s \in G$ is a semisimple element and $K$ is the subgroup
generated by $s$, then $K$ is linearly reductive, hence $G$-cr.
In this case, the groups $C_G(K)^0 = C_G(s)^0$ are the \emph{pseudo-Levi} subgroups of $G$,
which gives us \cite[Cor.\ 3.27]{FR}.
\end{rem}

\begin{cor}\label{cor:doublecent}
Suppose $H$ is a $G$-cr subgroup of $G$ such that $H = C_G(C_G(H)^0)^0$,
and $p$ is good for $H$ and $G$.
Then $\Omega_H^a(e) = \Omega_G^a(e) \cap Y(H)$ for all nilpotent elements $e \in \lieh$.
\end{cor}

\begin{proof}
This follows from Corollary \ref{cor:centcase}, setting $K = C_G(H)^0$.
\qed
\end{proof}

We finish this section with an example of how
Corollaries \ref{cor:centcase} and \ref{cor:doublecent}
can be applied.

\begin{exmp}
In \cite{ls} there are extensive tables of the subgroups of simple exceptional
algebraic groups and their centralizers. These tables can be used to generate many cases.
For example, looking at \cite[Table\ 8.1]{ls},
if $G = E_8$ and $p > 7$, then there is a pair of subgroups
$X_1 = G_2$ and $X_2 = F_4$ with $C_G(X_1)^0 = X_2$ and $C_G(X_2)^0 = X_1$.
Both these subgroups are $G$-cr, by \cite[Thm.\ 1]{ls}, so Corollary
\ref{cor:centcase} applies.
This is also an example of Corollary \ref{cor:doublecent},
since $C_G(C_G(X_i)^0)^0 = X_i$ for $i=1,2$.
\end{exmp}

\section{Extension to Non-Connected Groups}\label{sec:nonconnected}

In \cite[Sec.\ 6]{BMR}, and the later paper \cite{BMR2},
results about $G$-complete reducibility are proved in the more general case
where $G$ is reductive, but not necessarily connected;
we refer the reader to \cite[Sec.\ 6]{BMR} for the formalities.
Kempf's results can be immediately translated into this setting also,
so the results in Section 4 are easily extended.

In this section we briefly
indicate how to extend our results from Section \ref{sec:nilpotent}.
Since for any linear algebraic group $G$, $Y(G) = Y(G^0)$ and
$\Lie(G) = \Lie(G^0)$, moving to non-connected $G$ is not difficult.
To show what we mean, we give one possible extension of Corollary \ref{cor:centcase};
the proof comes from obvious extensions of our earlier proofs.

\begin{prop}\label{prop:nonconnected}
Let $G$ be a (possibly non-connected) reductive linear algebraic group,
and let $K$ be a $G$-completely reducible subgroup of $G$.
Set $H:= C_G(K)$ and
suppose that $p$ is good for $G^0$ and $H^0$.
Then $\Omega_H^a(e) = \Omega_{H^0}^a(e) = \Omega_G^a(e) \cap Y(H)$ for all nilpotent $e \in \lieh$.
\end{prop}

\begin{rem}
Proposition \ref{prop:nonconnected}
applies in particular in the case where $G = G_1 \rtimes K$,
where $G_1$ is a connected reductive group and $K$ is
a reductive group acting on $G_1$ by automorphisms
such that the image of $K$ in $G$ is $G$-cr.
Note that, by \cite[Cor.\ 14.11]{borel}, $G$ \emph{is} a reductive group
in this situation.
This allows one to consider outer automorphisms of a connected reductive
group, for example graph automorphisms of simple groups in good characteristic,
or automorphisms permuting the simple factors of a reductive group;
we give a simple illustration in Example \ref{exmp:diagonal} below.
\end{rem}

\begin{exmp}\label{exmp:diagonal}
Let $X$ be a connected reductive algebraic group such that $p$ is good for $X$.
Let $G_1 = X \times \ldots \times X$ be the direct product of $r$ copies of $X$,
and let $a$ be an automorphism of $G_1$ which acts as an $r$-cycle permuting the factors.
Set $K = \langle a \rangle$ and $G = G_1 \rtimes K$;
then $C_G(K)^0$ is the diagonal embedding of $X$ in $G_1$.
If $p$ divides $r$, then $K$ is not linearly reductive;
however, the image of $K$ in $G$ \emph{is} always $G$-cr.
(In the language of \cite[Sec.\ 6]{BMR},
any R-parabolic subgroup $P$ of $G$ which contains $K$ must be
of the form $P = (Q \times \ldots \times Q) \rtimes K$,
where $Q$ is a parabolic subgroup of $X$.
Then for any Levi subgroup $M$ of $Q$, $K$ is contained in
the R-Levi subgroup $L = (M\times\ldots\times M)\rtimes K$ of $P$.)
Thus Proposition \ref{prop:nonconnected} says that the diagonal embedding of a cocharacter
of $X$ associated with some $e \in \Lie(X)$ is a cocharacter of $G$ (hence of $G_1$) associated
with the diagonal embedding of $e$ in $\Lie(G) = \Lie(G_1)$.
\end{exmp}

We finish by noting that Corollary \ref{cor:doublecent} holds for non-connected
$G$ and $H$, with the weaker hypothesis that $H^0 = C_G(C_G(H))^0$, assuming $p$ is good for
$G^0$ and $H^0$.


\bigskip {\bf Acknowledgements}:
I am grateful to Gerhard R\"ohrle for introducing me to the problems discussed
in Section \ref{sec:nilpotent}, and for helpful discussions on earlier versions
of the paper.
Ben Martin has also given me valuable insight into the niceties of
Kempf's paper \cite{kempf}.
I would like to thank the referees for their pertinent comments,
which helped me to improve the clarity of the exposition.



\begin{thebibliography}{BMR2}
\bibitem[BMR]{BMR}
{M.~Bate, B.~Martin, G.~R\"ohrle},
{\it A Geometric Approach to Complete Reducibility},
Inv. math. {\bf 161}, no. 1 (2005), 177--218.

\bibitem[BMR2]{BMR2}
{M.~Bate, B.~Martin, G.~R\"ohrle},
{\it Complete Reducibility and Commuting Subgroups},
J. Reine Angew. Math., to appear.

\bibitem[B]{borel}
{A.~Borel},
{\it Linear Algebraic Groups},
Graduate Texts in Mathematics {\bf 126}, Springer-Verlag 1991.

\bibitem[FR]{FR}
{R.~Fowler, G.~R\"ohrle},
{\it On Cocharacters Associated to Nilpotent Elements of Reductive Groups},
Nagoya Math. J., to appear.

\bibitem[H]{hess}
{W.H.~Hesselink},
{\it Uniform Instability in Reductive Groups},
J. Reine Angew. Math. {\bf 303/304} (1978), 74--96.

\bibitem[J]{jantzen}
{J.C.~Jantzen},
{\it Nilpotent Orbits in Representation Theory},
in {\it Lie Theory. Lie Algebras and Representations},
Progress in Math. vol. {\bf 228}, J.-P. Anker, B. Orsted, eds., Birkh¨auser Boston, 2004.

\bibitem[K]{kempf}
{G.R.~Kempf},
{\it Instability in Invariant Theory},
Ann. Math. {\bf 108} (1978), 299--316.

\bibitem[LS]{ls}
{M.W.~Liebeck, G.M.~Seitz},
{\it Reductive subgroups of exceptional algebraic groups},
Mem. Amer. Math. Soc. no. {\bf 580} (1996).

\bibitem[L]{luna}
{D.~Luna},
{\it Adh\'erences d'orbite et invariants},
Inv. math. {\bf 29}, no. 3 (1975), 231--238.

\bibitem[M]{mcninch}
{G.~McNinch},
{\it Nilpotent Orbits over Ground Fields of Good Characteristic},
Math. Ann. {\bf 329}, no. 1 (2004), 49--85.

\bibitem[N]{nagata}
{M.~Nagata},
{\it Complete reducibility of rational representations of a matric group},
J. Math. Kyoto University {\bf 1} (1961), 87--99.

\bibitem[P]{premet}
{A.~Premet},
{\it Nilpotent Orbits in Good Characteristic and the Kempf-Rousseau Theory},
in {\it Special Edition celebrating the 80th birthday of Robert Steinberg},
J. Algebra {\bf 260}, no. 1 (2003), 338--366.

\bibitem[Ri]{rich0}
{R.W.~Richardson},
{\it On orbits of algebraic groups and Lie groups},
Bull. Austral. Math. Soc. {\bf 25}, no. 1 (1982),  1--28.

\bibitem[Ro]{rou}
{G.~Rousseau},
{\it Immeubles sph\'eriques et th\'eorie des invariants},
C.R.A.S. {\bf 286} (1978), 247--250.

\bibitem[S1]{serre1}
{J.-P. Serre},
{\it The notion of complete reducibility in group theory},
Moursund Lectures, Part II, University of Oregon, 1998.

\bibitem[S2]{serre2}
{J.-P. Serre},
{\it Compl\`ete R\'eductibilit\'e},
S\'eminaire Bourbaki, 56\`eme ann\'ee, 2003-2004, n$^{\rm o}$ {\bf 932}.

\bibitem[SS]{SS}
{T.A.~Springer, R.~Steinberg},
{\it Conjugacy classes},
in {\it Seminar on algebraic groups and related finite groups},
Lecture Notes in Mathematics {\bf 131},
Springer-Verlag, Heidelberg (1970), 167--266.
\end{thebibliography}
\end{document}